\documentclass[12pt]{amsart}
\usepackage{amsfonts}
\usepackage{amssymb,latexsym}
\usepackage{amsmath}
\usepackage{amsthm}
\usepackage{amssymb}
\usepackage{enumerate}
\newtheorem{theorem}{Theorem}[section]
\newtheorem{lemma}[theorem]{Lemma}
\newtheorem{proposition}[theorem]{Proposition}

\theoremstyle{definition}

\newtheorem{example}[theorem]{Example}

\theoremstyle{remark}
\newtheorem{remark}[theorem]{Remark}
\numberwithin{equation}{section}

\begin{document}
\setlength{\baselineskip}{1.2\baselineskip}
\title [$C^2$ interior a priori estimate ]
{The interior $C^2$ estimate for prescribed Gauss curvature equation in dimension two}
\author{Chuanqiang Chen}
\address{Department of Applied Mathematics,\\
         Zhejiang University of Technology\\
        Hangzhou, 310023, Zhejiang Province, CHINA}
\email{cqchen@mail.ustc.edu.cn}
\author{Fei Han}
\address{School of Mathematics Sciences\\
         Xinjiang Normal University\\
         Urumqi, 830054, Xinjiang Uygur Autonomous Region, CHINA}
\email{klhanqingfei@126.com}
\author{Qianzhong Ou}
\address{Department of Mathematics\\
          Hezhou University\\
         Hezhou, 542800, Guangxi Province, CHINA}
\email{ouqzh@163.com}
\thanks{Research of the first author is supported by the National Natural Science Foundation of China (NO. 11301497 and NO. 11471188). Research of the second author is supported by the National Natural Science Foundation of China (NO. 11161048). Research of the third author is supported by NSFC No. 11061013 and by Guangxi Science Foundation (2014GXNSFAA118028) and Guangxi Education Department Key Laboratory of Symbolic Computation and Engineering Data Processing.}
\maketitle

\begin{abstract}
In this paper, we introduce a new auxiliary function, and establish the interior $C^2$ estimate for prescribed Gauss curvature equation in dimension two.
\end{abstract}

\section{Introduction}

Given a positive function $f(x) \in C^2(\Omega)$ with $\Omega \subset \mathbb{R}^n$, to find a convex solution $u$, such that the Gauss curvature of the graph $(x,u(x))$ is $f(x)$, that is
\begin{equation}\label{1.1}
\frac{\det \nabla^2 u}{(1+ |\nabla u|^2)^{\frac{n+2}{2}}}=f(x), \quad \text{ in } \Omega.
\end{equation}
This is the classical prescribed Gauss curvature problem, and \eqref{1.1} is called prescribed Gauss curvature equation.

The a priori estimates are very important for fully nonlinear elliptic equations, especially the $C^2$ estimate. The interior $C^2$ estimate of $\sigma_2$ Hessian equation
\begin{align}\label{1.2*}
\sigma_2(\nabla^2 u) =f(x), \quad \text{ in }  B_R(0) \subset \mathbb{R}^n
 \end{align}
 in higher dimensions is a longstanding problem, where
$\sigma_2(\nabla^2u)=\sigma_2(\lambda(\nabla^2 u) )= \sum\limits_{1 \le i_1 < i_2 \leq n}\lambda_{i_1}\lambda_{i_2}$, $\lambda(\nabla^2 u)=(\lambda_1, \cdots, \lambda_n)$ are the eigenvalues of $\nabla^2 u$, and $f >0$. For $n=2$, \eqref{1.2*} is the Monge-Amp\`{e}re equation, and Heinz \cite{H59} obtained the estimate by the convex hypersurface geometry method (see \cite{CHO15} for an elementary analytic proof). For Monge-Amp\`{e}re equations with dimension $n \geq 3$, Pogorelov \cite{P78} constructed his famous counterexample, namely irregular solutions to Monge-Amp\`{e}re equations. For $n =3$ and $f \equiv 1$, \eqref{1.2*} can be reduced to a special Lagrangian equation after a Lewy transformation, and Warren-Yuan obtained the corresponding interior $C^2$ estimate in the celebrated paper \cite{WY09}. Moreover, the problem is still open for general $f$ with $n \geq 4$ and nonconstant $f$ with $n =3$. Also Urbas \cite{U90} generalized the counterexample for $\sigma_k$ Hessian equations with $k \geq 3$.

Moreover, Pogorelov type estimates for the Monge-Amp\`{e}re equations and $\sigma_k$ Hessian equation ($k \geq 2$) were derived by Pogorelov \cite{P78} and Chou-Wang \cite{CW01} respectively, and see \cite{GRW14} and \cite{LRW15} for some generalizations.

In this paper, we consider the convex solution of  prescribed Gauss curvature equation in dimension $n=2$ as follows
\begin{equation}\label{1.2}
\frac{\det \nabla^2 u}{(1+ |\nabla u|^2)^{2}}=f(x), \quad \text{ in }  B_R(0) \subset \mathbb{R}^2,
\end{equation}
and establish the interior $C^2$ estimate as follows
\begin{theorem}\label{th1.1}
Suppose $u \in C^4(B_R(0))$ be a convex solution of  prescribed Gauss curvature equation \eqref{1.2} in dimension $n=2$, where $0 < m \leq f \leq M$ in $B_R(0)$. Then
\begin{align}\label{1.3}
|\nabla^2 u(0)| \leq C(m, M, R, \sup |\nabla f|, \sup |\nabla^2 f|, \sup |\nabla u|),
\end{align}
where $C$ is a positive constant depending only on $m$, $M$, $R$, $\sup |\nabla f|$, $\sup |\nabla^2 f|$, and $\sup |\nabla u|$.
\end{theorem}

\begin{remark}
Heinz established an interior $C^2$ estimate for general Monge-Amp\`{e}re equation with general conditions in \cite{H59}, and the proof depends on the strict convexity of solutions and the geometry of convex hypersurface in dimension two. In this paper, our proof, which is based on a suitable choice of auxiliary functions, is elementary and avoids geometric computations on the graph of solutions. This technique is from \cite{CHO15}.
\end{remark}

\begin{remark}
By Trudinger's gradient estimates of Hessian equations in \cite{T97} or a gradient estimate of convex function, that is
\begin{align*}
\mathop {\sup }\limits_{B_{\frac{R}{2}}(0)} |\nabla u| \leq \frac{ 2 \mathop { osc }\limits_{B_{R}(0)} u}{R},
\end{align*}
we can bound $|\nabla^2 u(0)|$ in terms of $u$.
\end{remark}

The rest of the paper is organized as follows. In Section 2, we give the calculations of the derivatives of eigenvalues and eigenvectors with respect to the matrix. In Section 3, we introduce a new auxiliary function, and prove Theorem \ref{th1.1}.

\section{Derivatives of eigenvalues and eigenvectors}

In this section, we give the calculations of the derivatives of eigenvalues and eigenvectors with respect to the matrix. We think the following result is known for many people, for example see \cite{A07} for a similar result. For completeness, we give the result and a detailed proof.

\begin{proposition}\label{prop2.1}
Let $W= \{W_{ij}\}$ is an $n \times n$ symmetric matrix and $ \lambda(W)= (\lambda_1,\lambda _2, \cdots ,\lambda _{n})$ are the eigenvalues
of the symmetric matrix $W$, and the corresponding continuous eigenvector field is $\tau^i=({\tau^i}_{,1}, \cdots, {\tau^i}_{,n}) \in \mathbb{S}^{n-1}$. Suppose that $W= \{W_{ij}\}$ is diagonal, $\lambda_i= W_{ii}$ and the corresponding eigenvector $\tau^i=(0, \cdots ,0,\mathop 1\limits_{i^{th} } ,0, \cdots ,0) \in \mathbb{S}^{n-1}$ at the diagonal matrix $W$. If $\lambda_k$ is distinct with other eigenvalues, then we have at the diagonal matrix $W$
\begin{align}
\label{2.1}&\frac{{\partial {\tau^k} _{,k} }}{{\partial W_{pq} }} = 0, ~\forall ~p, q; \quad \frac{{\partial {\tau^k} _{,i} }}{{\partial W_{ik} }} = \frac{1}{\lambda_k - \lambda_i}, \quad i \ne k; \quad \frac{{\partial {\tau^k}_{,i} }}{{\partial W_{pq} }} = 0, \text{  otherwise}. &\\
\label{2.2}& \frac{{\partial ^2 {\tau^k}_{,k} }}
{{\partial W_{pk} \partial W_{pk} }} =  - \frac{1} {{(\lambda _k  - \lambda _p )^2 }}, \quad p \ne k;\\
\label{2.3}&\frac{{\partial ^2 {\tau^k} _{,i} }} {{\partial W_{ik} \partial W_{ii} }} = \frac{1}
{{(\lambda _k  - \lambda _i )^2 }}, ~i \ne k; \quad \frac{{\partial ^2 {\tau^k}_{,i} }}
{{\partial W_{ik} \partial W_{kk} }} =  - \frac{1}{{(\lambda _k  - \lambda _i )^2 }}, \quad i \ne k;\\
\label{2.4}& \frac{{\partial ^2 {\tau^k}_{,i} }}{{\partial W_{iq} \partial W_{qk} }} = \frac{1}
{{\lambda _k  - \lambda _i }}\frac{1}{{\lambda _k  - \lambda _q }}, \quad i \ne k,i \ne q,q \ne k;\\
\label{2.5}& \frac{{\partial ^2 {\tau^k}_{,i} }} {{\partial W_{pq} \partial W_{rs} }} = 0, \text{ otherwise}.
\end{align}
\end{proposition}

\begin{proof}
From the definition of eigenvalue and eigenvector of matrix $W$, we have
\begin{align*}
(W - \lambda _k I){\tau^k} \equiv 0,
\end{align*}
where ${\tau^k} $ is the eigenvector of $W$ corresponding to the eigenvalue $\lambda_k$. That is, for $i = 1, \cdots, n$, it holds
\begin{align}\label{2.6}
  [W_{ii}  - \lambda _k ]{\tau^k} _{,i}  + \sum\limits_{j \ne i} {W_{ij} } {\tau^k} _{,j}  = 0.
 \end{align}

When $W= \{W_{ij}\}$ is diagonal and $\lambda_k$ is distinct with other eigenvalues, $\lambda_k$ and $\tau^k$ are $C^2$ at the matrix $W$. In fact,
\begin{align}\label{2.7}
{\tau^k}_{,k}=1, \quad {\tau^k}_{,i} =0, \quad i \ne k, \quad \text{ at } W.
\end{align}

Taking the first derivative of \eqref{2.6}, we have
\begin{align*}
  [\frac{{\partial W_{ii} }}
{{\partial W_{pq} }} - \frac{{\partial \lambda _k }}
{{\partial W_{pq} }}]{\tau^k} _{,i}  + [W_{ii}  - \lambda _k ]\frac{{\partial {\tau^k}_{,i} }}
{{\partial W_{pq} }} + \sum\limits_{j \ne i} {[\frac{{\partial W_{ij} }}
{{\partial W_{pq} }}{\tau^k}_{,j}  + W_{ij} \frac{{\partial {\tau^k} _{,j} }}
{{\partial W_{pq} }}]}  = 0.
 \end{align*}
Hence for $i=k$, we get from \eqref{2.7}
\begin{align}\label{2.8}
  \frac{{\partial \lambda _k }}
{{\partial W_{pq} }} = \frac{{\partial W_{kk} }}
{{\partial W_{pq} }} = \left\{ \begin{gathered}
  \begin{array}{*{20}c}
   {1,} & {p = k,q = k;}  \\
 \end{array}  \hfill \\
  \begin{array}{*{20}c}
   {0,} & {otherwise.}  \\
 \end{array}  \hfill \\
\end{gathered}  \right.
 \end{align}
And for $i \ne k$,
\begin{align*}
  [W_{ii}  - \lambda _k ]\frac{{\partial {\tau^k}_{,i} }}
{{\partial W_{pq} }} + \sum\limits_{j \ne i} {\frac{{\partial W_{ij} }}
{{\partial W_{pq} }}{\tau^k}_{,j} }  = 0,
 \end{align*}
then
\begin{align}\label{2.9}
  \frac{{\partial {\tau^k}_{,i} }}
{{\partial W_{pq} }} = \frac{1}
{{\lambda _k  - \lambda _i }}\frac{{\partial W_{ik} }}
{{\partial W_{pq} }} = \left\{ \begin{gathered}
  \begin{array}{*{20}c}
   {\frac{1}
{{\lambda _k  - \lambda _i }},} & {p = i,q = k;}  \\
 \end{array}  \hfill \\
  \begin{array}{*{20}c}
   {0,} & {\qquad otherwise.}  \\
\end{array}  \hfill \\
\end{gathered}  \right.
 \end{align}

Since $\tau^k \in \mathbb{S}^{n-1}$, we have
\begin{align} \label{2.10}
  1 = |\tau ^k |^2  = ({\tau^k}_{,1} )^2  +  \cdots  + ({\tau^k}_{,k} )^2  +  \cdots  + ({\tau^k}_{,n} )^2.
\end{align}
Taking the first derivative of \eqref{2.10}, and using \eqref{2.7}, it holds
\begin{align}\label{2.11}
  \frac{{\partial {\tau^k}_{,k} }}{{\partial W_{pq} }} = 0, \quad \forall (p,q).
\end{align}

For $i=k$, taking the second derivative of \eqref{2.6}, and using \eqref{2.7}, it holds
\begin{align*}
  [\frac{{\partial ^2 W_{kk} }}
{{\partial W_{pq} \partial W_{rs} }} - \frac{{\partial ^2 \lambda _k }}
{{\partial W_{pq} \partial W_{rs} }}]{\tau^k}_{,k}  + \sum\limits_{j \ne k} {[\frac{{\partial W_{kj} }}
{{\partial W_{pq} }}\frac{{\partial {\tau^k}_{,j} }}
{{\partial W_{rs} }} + \frac{{\partial W_{kj} }}
{{\partial W_{rs} }}\frac{{\partial {\tau^k} _{,j} }}
{{\partial W_{pq} }}]}  = 0,
 \end{align*}
hence
\begin{align} \label{2.12}
  \frac{{\partial ^2 \lambda _k }}
{{\partial W_{pq} \partial W_{rs} }} =& \sum\limits_{j \ne k} {[\frac{{\partial W_{kj} }}
{{\partial W_{pq} }}\frac{{\partial {\tau^k} _{,j} }}
{{\partial W_{rs} }} + \frac{{\partial W_{kj} }}
{{\partial W_{rs} }}\frac{{\partial {\tau^k} _{,j} }}
{{\partial W_{pq} }}]}  \notag \\
 =& \left\{ \begin{gathered}
  \begin{array}{*{20}c}
   {\frac{1}
{{\lambda _k  - \lambda _q }},} & {p = k,q \ne k,r = q,s = k;}  \\
 \end{array}  \hfill \\
  \begin{array}{*{20}c}
   {\frac{1}
{{\lambda _k  - \lambda _s }},} & {r = k,s \ne k,p = s,q = k;}  \\
 \end{array}  \hfill \\
  \begin{array}{*{20}c}
   {0,} & {\qquad otherwise.}  \\
 \end{array}  \hfill \\
\end{gathered}  \right.
 \end{align}
For $i \ne k$, it holds
\begin{align*}
 & [\frac{{\partial W_{ii} }}
{{\partial W_{pq} }} - \frac{{\partial \lambda _k }}
{{\partial W_{pq} }}]\frac{{\partial {\tau^k}_{,i} }}
{{\partial W_{rs} }} + [\frac{{\partial W_{ii} }}
{{\partial W_{rs} }} - \frac{{\partial \lambda _k }}
{{\partial W_{rs} }}]\frac{{\partial {\tau^k}_{,i} }}
{{\partial W_{pq} }} + [W_{ii}  - \lambda _k ]\frac{{\partial ^2 {\tau^k}_{,i} }}
{{\partial W_{pq} \partial W_{rs} }}  \\
&+ \sum\limits_{j \ne i} {[\frac{{\partial W_{ij} }}
{{\partial W_{pq} }}\frac{{\partial {\tau^k}_{,j} }}
{{\partial W_{rs} }} + \frac{{\partial W_{ij} }}
{{\partial W_{rs} }}\frac{{\partial {\tau^k}_{,j} }}
{{\partial W_{pq} }}]}  = 0,
 \end{align*}
then
\begin{align}
\label{2.13}&\frac{{\partial ^2 {\tau^k}_{,i} }}{{\partial W_{ik} \partial W_{ii} }} = \frac{1}
{{\lambda _k  - \lambda _i }}\frac{{\partial {\tau^k}_{,i} }}{{\partial W_{ik} }} = \frac{1}
{{\lambda _k  - \lambda _i }}\frac{1}{{\lambda _k  - \lambda _i }}, \quad i \ne k;   \\
\label{2.14}&\frac{{\partial ^2 {\tau^k}_{,i} }}{{\partial W_{iq} \partial W_{qk} }} = \frac{1}
{{\lambda _k  - \lambda _i }}\frac{{\partial W_{iq} }}
{{\partial W_{iq} }}\frac{{\partial {\tau^k}_{,q} }}
{{\partial W_{qk} }} = \frac{1}{{\lambda _k  - \lambda _i }}\frac{1}
{{\lambda _k  - \lambda _q }}, \quad i \ne k,i \ne q,q \ne k; \\
\label{2.15}&\frac{{\partial ^2 {\tau ^k} _{,i} }} {{\partial W_{ik} \partial W_{kk} }} = \frac{1}
{{\lambda _k  - \lambda _i }}[ - \frac{{\partial \lambda _k }}
{{\partial W_{kk} }}\frac{{\partial {\tau ^k} _{,i} }}{{\partial W_{ik} }}] =  - \frac{1}
{{\lambda _k  - \lambda _i }}\frac{1}{{\lambda _k  - \lambda _i }}, \quad i \ne k; \\
\label{2.16}&\frac{{\partial ^2 {\tau^k}_{,i} }}{{\partial W_{pq} \partial W_{rs} }} = 0, \quad \text{ otherwise. }
\end{align}

From \eqref{2.10}, we have
\begin{align*}
2\tau ^k _{,k} \frac{{\partial ^2 {\tau^k}_{,k} }}
{{\partial W_{pq} \partial W_{rs} }} + 2\sum\limits_{i \ne k} {\frac{{\partial {\tau^k}_{,i} }}
{{\partial W_{pq} }}\frac{{\partial {\tau^k}_{,i} }}
{{\partial W_{rs} }}}  = 0,
 \end{align*}
then
\begin{align*}
\frac{{\partial ^2 {\tau^k}_{,k} }}
{{\partial W_{pq} \partial W_{rs} }} =  - \sum\limits_{i \ne k} {\frac{{\partial {\tau^k}_{,i} }}
{{\partial W_{pq} }}\frac{{\partial {\tau^k}_{,i} }}
{{\partial W_{rs} }}}  = \left\{ \begin{gathered}
  \begin{array}{*{20}c}
   { - \frac{1}
{{\lambda _k  - \lambda _p }}\frac{1}
{{\lambda _k  - \lambda _p }},} & {p \ne k,q = k,r = p,s = q;}  \\
 \end{array}  \hfill \\
  \begin{array}{*{20}c}
   {0,} & {\qquad \qquad \quad otherwise.}  \\
 \end{array}  \hfill \\
\end{gathered}  \right.
\end{align*}
The proof of Proposition \ref{prop2.1} is finished.
\end{proof}

\begin{example}
When $n =2$, the matrix $
\left( {\begin{array}{*{20}c}
   {u_{11} } & {u_{12} }  \\
   {u_{21} } & {u_{22} }  \\

 \end{array} } \right)
$
has two eigenvalues
\[
\begin{gathered}
  \lambda _1  = \frac{{(u_{11}  + u_{22} ) + \sqrt {(u_{11}  - u_{22} )^2  + 4u_{12} u_{21} } }}
{2}, \hfill \\
  \lambda _2  = \frac{{(u_{11}  + u_{22} ) - \sqrt {(u_{11}  - u_{22} )^2  + 4u_{12} u_{21} } }}
{2}, \hfill \\
\end{gathered}
\]
with $\lambda_1 \geq \lambda_2$. If $\lambda_1 > \lambda_2$,
\[
\left[ {\left( {\begin{array}{*{20}c}
   {u_{11} } & {u_{12} }  \\
   {u_{21} } & {u_{22} }  \\

 \end{array} } \right) - \lambda _1 \left( {\begin{array}{*{20}c}
   1 & 0  \\
   0 & 1  \\

 \end{array} } \right)} \right]\left( {\begin{array}{*{20}c}
   {\xi _1 }  \\
   {\xi _2 }  \\

 \end{array} } \right) = 0,
\]
we can get
\[
\begin{gathered}
  \xi _1  = \frac{{(u_{22}  - u_{11} ) - \sqrt {(u_{11}  - u_{22} )^2  + 4u_{12} u_{21} } }}
{2}; \hfill \\
  \xi _2  =  - u_{21}. \hfill \\
\end{gathered}
\]
Then the eigenvector $\tau$ corresponding to $\lambda_1$ is
\[
\tau  =  - \frac{{(\xi _1 ,\xi _2 )}}
{{\sqrt {{\xi _1 }^2  + {\xi _2} ^2 } }}.
\]
We can verify Proposition \ref{prop2.1}.
\end{example}

\section{Proof of Theorem \ref{th1.1}}

Now we start to prove Theorem \ref{th1.1}.

Let $\tau (x) = \tau(\nabla^2 u(x)) =(\tau_1, \tau_2) \in \mathbb{S}^{1}$ be the continuous eigenvector field of $\nabla^2 u(x)$ corresponding to the largest eigenvalue. Denote
\begin{align}\label{3.2}
\Sigma =: \{x \in B_R(0): r^2 -|x|^2 + \langle x, \tau(x)\rangle^2 >0, r^2 - \langle x, \tau(x)\rangle^2 >0\},
\end{align}
where $r = \frac{1}{\sqrt 2} R$. It is easy to know, $\Sigma$ is an open set and $ B_{r}(0) \subset \Sigma \subset  B_R(0)$.
We introduce a new auxiliary function in $\Sigma$ as follows
\begin{equation} \label{3.1}
\phi(x)  = \eta(x)^\beta g(\frac{1}{2}|Du|^2 )u_{\tau \tau}
\end{equation}
where $\eta (x) = (r^2 -|x|^2 + \langle x, \tau(x)\rangle^2 )(r^2 - \langle x, \tau(x)\rangle^2)$ with $\beta = 4$ and $g(t) = e^{\frac{c_0}{r^2}t}$ with $c_0 = \frac{32}{m}$. In fact, $\langle x, \tau(x)\rangle$ is invariant under rotations of the coordinates, so is $\eta(x)$.

From the definition of $\Sigma$, we know $\eta(x) >0$ in $\Sigma$, and $\eta =0 $ on $\partial \Sigma$. Assume the maximum of $\phi(x)$ in $\Sigma$ is attains at $x_0 \in \Sigma$. By rotating the coordinates, we can assume $\nabla^2u (x_0)$ is diagonal. In the following, we denote $\lambda_i = u_{ii}(x_0)$, $\lambda = (\lambda_1, \lambda_2)$. Without loss of generality, we can assume $\lambda_1 \geq \lambda_2 $. Then $\tau(x_0) = (1, 0)$.

We will use notion $h=O(f)$ if
$|h(x)| \le Cf(x)$ for any $x \in \Omega$ with a positive constant $C$ depending only on $m$, $M$, $R$, $\sup |\nabla f|$, $\sup |\nabla^2 f|$, and $\sup |\nabla u|$. Similarly we write $h \geq O(f)$ if $h(x) \geq  - C f(x)$ and $h \leq O(f)$ if $h(x) \leq  C f(x)$.

Now, we assume $\eta \lambda_1$ is big enough. Otherwise there is nothing to prove. Then we have from the equation \eqref{1.2},
\begin{align}\label{3.3}
\lambda_2 = \frac{f (1+ |\nabla u|^2)^{2}}{\lambda_1} \leq  \frac{M (1+ |\nabla u|^2)^{2}}{\lambda_1} < \lambda_1.
\end{align}
Hence $\lambda_1 $ is distinct with the other eigenvalue, and $\tau (x)$ is $C^2$ at $x_0$. Moreover, the test function
\begin{equation} \label{3.4}
\varphi  = \beta \log \eta  + \log g(\frac{1}{2}|\nabla u|^2 ) + \log u_{11}
\end{equation}
attains the local maximum at $x_0$. In the following, all the calculations are at $x_0$. Then, we can get
\begin{equation}
0= \varphi _i  = \beta \frac{{\eta _i }}{\eta } + \frac{{g'}}{g}\sum\limits_k
{u_k u_{ki} }  + \frac{{u_{11i} }}{{u_{11} }} , \notag
\end{equation}
so we have
\begin{equation} \label{3.5}
\frac{{u_{11i} }}{{u_{11} }} =  - \beta  \frac{{\eta _i }}{\eta } -
\frac{{g'}}{g}u_i u_{ii}, \quad i =1, 2.
\end{equation}
At $x_0$, we also have
\begin{align*}
0 \geq \varphi _{ii}  =& \beta [ \frac{{\eta _{ii} }}{\eta } - \frac{{\eta _i ^2
}}{{\eta ^2 }}] + \frac{{g''g - g'^2 }}{{g^2 }}\sum\limits_k {u_k
u_{ki} } \sum\limits_l {u_l u_{li} } \notag \\
&+ \frac{{g'}}{g}\sum\limits_k {(u_{ki} u_{ki}  + u_k u_{kii} )}  +
\frac{{u_{11ii} }}{{u_{11} }} - \frac{{u_{11i} ^2 }}{{u_{11} ^2 }}
\notag \\
=& \beta  [\frac{{\eta _{ii} }}{\eta } - \frac{{\eta _i ^2 }}{{\eta ^2 }} ]+ \frac{{g'}}{g}[u_{ii} ^2  + \sum\limits_k {u_k u_{kii} } ] +
\frac{{u_{11ii} }}{{u_{11} }} - \frac{{u_{11i} ^2 }}{{u_{11} ^2 }},
\end{align*}
since $g''g - g'^2 =0$. Let
\begin{align*}
&F^{11} = \frac{\partial \det \nabla^2 u}{\partial u_{11}} = \lambda_2, \quad F^{22} = \frac{\partial \det \nabla^2 u}{\partial u_{22}} = \lambda_1, \\
&F^{12} = \frac{\partial \det \nabla^2 u}{\partial u_{12}} = 0, \quad F^{21} = \frac{\partial \det \nabla^2 u}{\partial u_{21}} = 0.
\end{align*}
Then from the equation \eqref{1.2} we can get
\begin{align}\label{3.6}
\lambda_2 = \frac{f (1+ |\nabla u|^2)^{2}}{\lambda_1}.
\end{align}
Differentiating \eqref{1.2} once, we can get
\begin{align}\label{3.7}
F^{11} u_{11i} +F^{22} u_{22i} = f_i (1+ |\nabla u|^2)^{2} + f \cdot 2(1+ |\nabla u|^2) \cdot 2 u_i u_{ii},
\end{align}
then
\begin{align}\label{3.8}
 u_{22i} =& \frac{1}{F^{22}}[f_i (1+ |\nabla u|^2)^{2} + f \cdot 2(1+ |\nabla u|^2) \cdot 2 u_i u_{ii} - F^{11} u_{11i} ] \notag\\
 =& \frac{f_i (1+ |\nabla u|^2)^{2}}{\lambda_1}+f \cdot 4(1+ |\nabla u|^2) u_i \frac{ u_{ii}}{\lambda_1} -\frac{f (1+ |\nabla u|^2)^{2}}{\lambda_1} \frac{u_{11i}}{u_{11}}\notag\\
 =& -\frac{f (1+ |\nabla u|^2)^{2}}{\lambda_1} \frac{u_{11i}}{u_{11}} + O(1).
\end{align}
Differentiating \eqref{1.2} twice, we can get
\begin{align}\label{3.9}
&F^{11} u_{1111} +F^{22} u_{2211}  \notag \\
=& \frac{\partial^2  }{ \partial x_1^2}[f (1+ |\nabla u|^2)^{2} ]- 2\frac{\partial^2 \det \nabla^2 u }{\partial u_{11} \partial u_{22}}u_{111}u_{221} - 2 \frac{\partial^2 \det \nabla^2 u }{\partial u_{12} \partial u_{21}}u_{112}^2  \notag \\
=& f_{11}(1+ |\nabla u|^2)^{2} + 2 f_1 \cdot 2(1+ |\nabla u|^2) \cdot 2 u_1 u_{11} \notag \\
&+ f[ 2 (2 u_1 u_{11})^2 +  2(1+ |\nabla u|^2) \cdot (2 u_{11}^2 + 2 u_k u_{k11})] \notag \\
&- 2u_{111}u_{221} + 2 u_{112}^2  \notag \\
=& f_{11}(1+ |\nabla u|^2)^{2} + 8 f_1 \cdot (1+ |\nabla u|^2) u_1 u_{11} +  f[8 u_1^2 +  4(1+ |\nabla u|^2)] u_{11}^2 \notag \\
&+ 4f (1+ |\nabla u|^2)[  u_1 u_{111}+ u_2 u_{211}] + 2 u_{112}^2 \notag \\
&- 2u_{111}[\frac{f_1 (1+ |\nabla u|^2)^{2}}{\lambda_1}+f \cdot 4(1+ |\nabla u|^2) u_1  -\frac{f (1+ |\nabla u|^2)^{2}}{\lambda_1} \frac{u_{111}}{u_{11}}]\notag \\
=& f_{11}(1+ |\nabla u|^2)^{2} + 8 f_1 \cdot (1+ |\nabla u|^2) u_1 u_{11} +  f[8 u_1^2 +  4(1+ |\nabla u|^2)] u_{11}^2 \notag \\
&+ 4f (1+ |\nabla u|^2)[  -u_1 u_{111}+ u_2 u_{211}] -2 f_1 (1+ |\nabla u|^2)^{2} \frac{u_{111}}{u_{11}} \notag \\
&+ 2 u_{112}^2 + 2 f(1+ |\nabla u|^2)^{2} (\frac{u_{111}}{u_{11}})^2\notag \\
=& f[8 u_1^2 +  4(1+ |\nabla u|^2)] u_{11}^2 + O(\lambda_{1}) \notag \\
&+ 4f (1+ |\nabla u|^2)[  -u_1 u_{111}+ u_2 u_{112}] -2 f_1 (1+ |\nabla u|^2)^{2} \frac{u_{111}}{u_{11}} \notag \\
&+ 2 u_{112}^2 + 2 f(1+ |\nabla u|^2)^{2} (\frac{u_{111}}{u_{11}})^2,
\end{align}
and
\begin{align}\label{3.10}
&F^{11} u_{1112} +F^{22} u_{2212} \notag \\
=& \frac{\partial^2  }{ \partial x_1 \partial x_2}[f (1+ |\nabla u|^2)^{2} ] \notag\\
&- \frac{\partial^2 \det \nabla^2 u }{\partial u_{11} \partial u_{22}}u_{111}u_{222}- \frac{\partial^2 \det \nabla^2 u }{\partial u_{22} \partial u_{11}}u_{221}u_{112}- 2\frac{\partial^2 \det \nabla^2 u }{\partial u_{12} \partial u_{21}} u_{121}u_{212}  \notag \\
=&f_{12}(1+ |\nabla u|^2)^{2} + f_1 \cdot 2(1+ |\nabla u|^2) \cdot 2 u_2 u_{22}  + f_2 \cdot 2(1+ |\nabla u|^2) \cdot 2 u_1 u_{11} \notag \\
&+ f[ 2 (2 u_1 u_{11})(2 u_2 u_{22}) +  2(1+ |\nabla u|^2) \cdot 2 u_k u_{k12}] \notag \\
& - u_{111}u_{222}- u_{112}u_{221} + 2 u_{112}u_{221}  \notag \\
=&f_{12}(1+ |\nabla u|^2)^{2} + 4(1+ |\nabla u|^2) [f_1 \cdot u_2 u_{22}  + f_2 \cdot u_1 u_{11}] \notag \\
&+ 8 f^2 u_1  u_2(1+ |\nabla u|^2)^{2} +  4 f (1+ |\nabla u|^2) u_1 u_{112} \notag \\
&+  4 f (1+ |\nabla u|^2) u_2 [  -\frac{f (1+ |\nabla u|^2)^{2}}{\lambda_1} \frac{u_{111}}{u_{11}} + O(1)] \notag \\
& - u_{111} [\frac{f_2 (1+ |\nabla u|^2)^{2}}{\lambda_1}+f \cdot 4(1+ |\nabla u|^2) u_2 \frac{ u_{22}}{\lambda_1} -\frac{f (1+ |\nabla u|^2)^{2}}{\lambda_1} \frac{u_{112}}{u_{11}}]  \notag \\
&+  u_{112}[\frac{f_1 (1+ |\nabla u|^2)^{2}}{\lambda_1}+f \cdot 4(1+ |\nabla u|^2) u_1  -\frac{f (1+ |\nabla u|^2)^{2}}{\lambda_1} \frac{u_{111}}{u_{11}}] \notag  \\
=& - \frac{u_{111}}{u_{11}} [f_2 (1+ |\nabla u|^2)^{2} + f \cdot 8(1+ |\nabla u|^2) u_2 u_{22} ]  \notag \\
&+  \frac{u_{112}}{u_{11}}[f_1 (1+ |\nabla u|^2)^{2} + f \cdot 8(1+ |\nabla u|^2) u_1 u_{11}] + O(\lambda_{1}).
\end{align}

Hence we can get by \eqref{3.7} and \eqref{3.9},
\begin{eqnarray} \label{3.11}
0 &\ge& \sum\limits_{i = 1}^2 {F^{ii} \varphi _{ii} }  \notag\\
&=& \beta \sum\limits_{i } {F^{ii} [\frac{{\eta _{ii} }}{\eta } -
\frac{{\eta _i ^2 }}{{\eta ^2 }}]}+ \frac{{g'}}{g}\sum\limits_{i } {F^{ii} u_{ii} ^2 } + \frac{{g'}}{g}\sum\limits_k {u_k \sum\limits_i F^{ii} u_{iik} }  \notag \\
&& +\frac{1}{{u_{11} }}\sum\limits_{i } {F^{ii} u_{11ii} }  - \sum\limits_{i } {F^{ii} [\frac{{u_{11i} }}{{u_{11} }}]^2 } \notag \\
&=& \beta \lambda_2 [\frac{{\eta _{11} }}{\eta } - \frac{{\eta _1 ^2 }}{{\eta ^2 }}] +\beta \lambda_1 [\frac{{\eta _{22}
}}{\eta } - \frac{{\eta _2 ^2 }}{{\eta ^2 }}] + \frac{{g'}}{g}[\lambda_1 + \lambda_2]f (1+ |\nabla u|^2)^{2}  \notag \\
&&+ \frac{{g'}}{g} [( u_1 f_1 + u_2  f_2 ) (1+ |\nabla u|^2)^{2} + 4 f \cdot (1+ |\nabla u|^2)( u_1^2 u_{11}+ u_2^2 u_{22}) ] \notag \\
&& + \frac{1}{{u_{11} }}\big\{ f[8 u_1^2 +  4(1+ |\nabla u|^2)] u_{11}^2 + O(\lambda_{1})\notag \\
&&\qquad \quad + 4f (1+ |\nabla u|^2)[  -u_1 u_{111}+ u_2 u_{211}] -2 f_1 (1+ |\nabla u|^2)^{2} \frac{u_{111}}{u_{11}} \notag \\
&&\qquad \quad  + 2 u_{112}^2 + 2 f(1+ |\nabla u|^2)^{2} (\frac{u_{111}}{u_{11}})^2\big\}  \notag \\
&&- \lambda_2 [\frac{{u_{111} }}{{u_{11} }}]^2  - \lambda_1 [\frac{{u_{112} }}{{u_{11} }}]^2 \notag \\
&\geq&  \beta [\lambda_2 \frac{{\eta _{11}}}{\eta } + \lambda_1 \frac{{\eta _{22}
}}{\eta } ] - \beta \lambda_2 \frac{{\eta _1 ^2 }}{{\eta ^2 }} - \beta \lambda_1 \frac{{\eta _2 ^2 }}{{\eta ^2 }} + \frac{{g'}}{g} f (1+ |\nabla u|^2)^{2} \lambda_1 \notag \\
&&+ \lambda_2  [\frac{{u_{111} }}{{u_{11} }}]^2  -[ 4f (1+ |\nabla u|^2)u_1 + 2 \frac{f_1 (1+ |\nabla u|^2)^{2}}{{u_{11} }} ]\frac{ u_{111}}{u_{11}} \notag \\
&&+  f[8 u_1^2 +  4(1+ |\nabla u|^2)] u_{11} \notag \\
&&+ \frac{ \lambda_1 }{2}[\frac{{u_{112} }}{{u_{11} }}]^2+ \frac{ \lambda_1 }{2}[\beta\frac{{\eta_{2} }}{{\eta }} + \frac{g'}{g} u_2 u_{22}]^2 +4f (1+ |\nabla u|^2)u_2 \frac{{u_{112} }}{{u_{11} }}+ O(1) \notag \\
&\geq&  \beta [\lambda_2\frac{{\eta _{11}}}{\eta } + \lambda_1 \frac{{\eta _{22}
}}{\eta } ] - \beta \lambda_2\frac{{\eta _1 ^2 }}{{\eta ^2 }}   + \frac{{g'}}{g} f (1+ |\nabla u|^2)^{2} \lambda_1 \notag \\
&&+ \frac{\lambda_2}{ 2} [\frac{{u_{111} }}{{u_{11} }}]^2 + \frac{ \lambda_1 }{2}[\frac{{u_{112} }}{{u_{11} }}]^2 + (\frac{ \beta^2 }{2} - \beta) \lambda_1 \frac{{\eta_{2}^2 }}{{\eta^2 }} \notag \\
&&+\beta f (1+ |\nabla u|^2)[(1+ |\nabla u|^2)\frac{g'}{g} -4 ]\frac{{\eta_{2} }}{{\eta }} u_2+ O(1),
\end{eqnarray}

where we used the following inequalities
\begin{align*}
&\lambda_2  [\frac{{u_{111} }}{{u_{11} }}]^2  -[ 4f (1+ |\nabla u|^2)u_1 + 2 \frac{f_1 (1+ |\nabla u|^2)^{2}}{{u_{11} }} ]\frac{ u_{111}}{u_{11}} +  f[8 u_1^2 +  4(1+ |\nabla u|^2)] u_{11} \\
=& \frac{\lambda_2}{2}  [\frac{{u_{111} }}{{u_{11} }}]^2 + \frac{\lambda_2}{2}\big \{\frac{ u_{111}}{u_{11}} -  \frac{1}{\lambda_2} [ 4f (1+ |\nabla u|^2)u_1 + 2 \frac{f_1 (1+ |\nabla u|^2)^{2}}{{u_{11} }} ] \big\}^2  \\
& - \frac{1}{2 \lambda_2} [ 4f (1+ |\nabla u|^2)u_1 + 2 \frac{f_1 (1+ |\nabla u|^2)^{2}}{{u_{11} }} ] ^2  +  f[8 u_1^2 +  4(1+ |\nabla u|^2)] u_{11} \\
\geq& \frac{\lambda_2}{2}  [\frac{{u_{111} }}{{u_{11} }}]^2  - \frac{\lambda_1}{2 f (1+ |\nabla u|^2)^{2}} [ 4f (1+ |\nabla u|^2)u_1 + 2 \frac{f_1 (1+ |\nabla u|^2)^{2}}{{u_{11} }} ] ^2 \\
& +  f[8 u_1^2 +  4(1+ |\nabla u|^2)] u_{11} \\
=& \frac{\lambda_2}{2}  [\frac{{u_{111} }}{{u_{11} }}]^2  - 8 f u_1^2  \lambda_1 -8 f_1 (1+ |\nabla u|^2) u_1 -   2 \frac{f_1^2 (1+ |\nabla u|^2)^{2}}{{f u_{11} }} \\
& +  f[8 u_1^2 +  4(1+ |\nabla u|^2)] u_{11} \\
\geq& \frac{\lambda_2}{2}  [\frac{{u_{111} }}{{u_{11} }}]^2 -8 f_1 (1+ |\nabla u|^2) u_1 -   2 \frac{f_1^2 (1+ |\nabla u|^2)^{2}}{{f u_{11} }},
\end{align*}
and 
\begin{align*}
4f (1+ |\nabla u|^2)u_2 \frac{{u_{112} }}{{u_{11} }} =& 4f (1+ |\nabla u|^2)u_2 [-\beta\frac{{\eta_{2} }}{{\eta }} - \frac{g'}{g} u_2 u_{22}]  \\
=& -4 \beta f (1+ |\nabla u|^2)\frac{{\eta_{2} }}{{\eta }}u_2 + O(1).
\end{align*}

Now we have the following lemma.

\begin{lemma}\label{lem3.1}
If $ \eta \lambda_1$ is big enough, we have at $x_0$
\begin{align}
\label{3.12}&\beta \lambda_2 \frac{{\eta _1 ^2 }}{{\eta ^2 }} \leq \frac{\lambda_1}{4}(\frac{u_{112}} {{u_{11}  }})^2 + O (\frac{1} {{ \eta }});   \\
\label{3.13}&\beta f (1+ |\nabla u|^2)[(1+ |\nabla u|^2)\frac{g'}{g} -4 ]\frac{{\eta_{2} }}{{\eta }} u_2  = O (\frac{1} {{ \eta }});
\end{align}
and
\begin{align}\label{3.14}
\beta [\lambda_2\frac{{\eta _{11}}}{\eta } + \lambda_1 \frac{{\eta _{22}
}}{\eta } ] \geq -  \frac{1}{2}\frac{{g'}}{g} f\lambda_1-\beta \lambda_1 [\frac{\eta_2}{\eta}]^2- \frac{\lambda_2}{2}(\frac{u_{111}}{u_{11}})^2 -\frac{\lambda_1}{4}(\frac{u_{112}} {{u_{11} }})^2+ O(\frac{{1}}{\eta} ).
\end{align}
\end{lemma}

\begin{proof}
At $x_0$, $\tau=(\tau_1, \tau_2)= (1, 0)$.
Then from Proposition \ref{prop2.1} we get
\begin{align}\label{3.15}
\left\langle {x,\partial _i \tau } \right\rangle  =&  \sum\limits_{m = 1}^2 {x_m \frac{{\partial \tau _m }}
{{\partial x_i }}}  = \sum\limits_{m = 1}^2 {x_m \frac{{\partial \tau _m }}
{{\partial u_{pq} }}u_{pqi}} = x_2 \frac{{\partial \tau _2 }}
{{\partial u_{pq} }}u_{pqi}   \notag \\
=& x_2 \frac{u_{12i}} {{\lambda_1 - \lambda_2 }}, \quad i = 1, 2.
\end{align}

From the definition of $\eta$, then we have at $x_0$
\begin{align}\label{3.16}
\eta  = [r^2  - |x|^2  + \left\langle {x,\tau } \right\rangle ^2 ][r^2  - \left\langle {x,\tau } \right\rangle ^2 ] =  (r^2  - x_2 ^2)(r^2  - x_1 ^2).
\end{align}
Taking the first derivative of $\eta$, we can get
\begin{align*}
\eta _i =& [ - 2x_i  + 2\left\langle {x,\tau } \right\rangle \left\langle {x,\tau } \right\rangle _i ][r^2  - \left\langle {x,\tau } \right\rangle ^2 ]\\
   &+ [r^2  - |x|^2  + \left\langle {x,\tau } \right\rangle ^2 ][ - 2\left\langle {x,\tau } \right\rangle \left\langle {x,\tau } \right\rangle _i ] \\
   =& [ - 2x_i  + 2x_1 ( \delta_{i1}+ \left\langle {x, \partial_i \tau } \right\rangle ) ][r^2  - x_1 ^2 ] +( r^2 -x_2^2) [ - 2x_1  ( \delta_{i1} + \left\langle {x, \partial_i \tau } \right\rangle )  ] \hfill \\
   =& \left\{ \begin{array}{l}
  - 2x_1(r^2  - x_2 ^2 )+  2x_1 \left\langle {x, \partial_1 \tau } \right\rangle (x_2^2 - x_1 ^2), \quad i =1;\\
  - 2x_2(r^2  - x_1 ^2 )+  2x_1 \left\langle {x, \partial_2 \tau } \right\rangle (x_2^2 - x_1 ^2), \quad i =2.\\
 \end{array} \right.
\end{align*}
Hence if $ \eta \lambda_1$ is big enough, we can get
\begin{align} \label{3.17}
\beta \lambda_2\frac{{\eta _1 ^2 }}{{\eta ^2 }} =&  \beta \lambda_2 [\frac{- 2x_1(r^2  - x_2^2)}{\eta} + 2x_1 x_2 \frac{x_2^2 - x_1 ^2}{\eta} \frac{u_{112}} {{\lambda_1 - \lambda_2 }}]^2  \notag \\
\leq& \beta \lambda_2 [\frac{8r^6 }{\eta^2} + \frac{8r^8 }{\eta^2} (\frac{u_{112}} {{\lambda_1 - \lambda_2  }})^2]\notag \\
\leq& \frac{\lambda_1}{4}(\frac{u_{112}} {{u_{11}  }})^2 + O (\frac{1} {{ \eta }}).
\end{align}
Also we have
\begin{align}
\frac{{\eta _2 }}{{\eta }} =&\frac{- 2x_2(r^2  - x_1^2)}{\eta} + 2x_1 x_2 \frac{x_2^2 - x_1 ^2}{\eta} \frac{u_{221}} {{\lambda_1 - \lambda_2 }} \notag \\
=&\frac{- 2x_2(r^2  - x_1^2)}{\eta} + 2x_1 x_2 \frac{x_2^2 - x_1 ^2}{\eta} \frac{1} {{\lambda_1 - \lambda_2 }} [  -\frac{f (1+ |\nabla u|^2)^{2}}{\lambda_1} \frac{u_{111}}{u_{11}} + O(1)]\notag \\
=&\frac{- 2x_2(r^2  - x_1^2)}{\eta} [1- x_1 \frac{(x_2^2 - x_1 ^2)(r^2 - x_2^2)}{\eta} \frac{1} {{\lambda_1 - \lambda_2 }} \cdot O(1) ] \notag \\
&+ 2x_1 x_2 \frac{x_2^2 - x_1 ^2}{\eta} \frac{1} {{\lambda_1 - \lambda_2 }} \frac{f (1+ |\nabla u|^2)^{2}}{\lambda_1}[ \beta \frac{{\eta _1 }}{{\eta }} + \frac{{g'}}{g}u_1 u_{11}]\notag \\
=&\frac{- 2x_2(r^2  - x_1^2)}{\eta} [1 + O (\frac{1} {{\eta \lambda_1 }})- x_1 \frac{(x_2^2 - x_1 ^2)(r^2 - x_2^2)}{\eta} \frac{1} {{\lambda_1 - \lambda_2 }} f (1+ |\nabla u|^2)^{2}\frac{{g'}}{g}u_1] \notag \\
&+ 2x_1 x_2 \frac{x_2^2 - x_1 ^2}{\eta} \frac{1} {{\lambda_1 - \lambda_2 }} \frac{f(1+ |\nabla u|^2)^{2}}{\lambda_1} \beta [\frac{- 2x_1(r^2  - x_2^2)}{\eta} + 2x_1 x_2 \frac{x_2^2 - x_1 ^2}{\eta} \frac{u_{112}} {{\lambda_1 - \lambda_2 }}]\notag \\
=&\frac{- 2x_2(r^2  - x_1^2)}{\eta} [1 + O (\frac{1} {{\eta \lambda_1 }})] \notag \\
&+[ 2x_1 x_2 \frac{x_2^2 - x_1 ^2}{\eta} ]^2 \frac{f(1+ |\nabla u|^2)^{2}} {{(\lambda_1 - \lambda_2 )^2}} \beta [ - \beta \frac{ \eta_{2}}{ \eta} - \frac{{g'}}{g}u_2 u_{22}], \notag
\end{align}
then we can get
\begin{align}\label{3.18}
&[1+ \beta ^2 ( 2x_1 x_2 \frac{x_2^2 - x_1 ^2}{\eta} )^2 \frac{f(1+ |\nabla u|^2)^{2}} {{(\lambda_1 - \lambda_2 )^2}} ]\frac{{\eta _2 }}{{\eta }} \notag \\
=& \frac{- 2x_2(r^2  - x_1^2)}{\eta} [1+ O (\frac{1} {{\eta \lambda_1 }})]  -[ 2x_1 x_2 \frac{x_2^2 - x_1 ^2}{\eta} ]^2 \frac{f(1+ |\nabla u|^2)^{2}} {{(\lambda_1 - \lambda_2 )^2}} \beta \frac{{g'}}{g}u_2 u_{22}  \notag  \\
=& \frac{- 2x_2(r^2  - x_1^2)}{\eta} [1+ O (\frac{1} {{\eta \lambda_1 }})+ 2x_1^2 x_2 \frac{(x_2^2 - x_1 ^2)^2}{\eta^2} \frac{(r^2  - x_2^2)}{\eta} \frac{f(1+ |\nabla u|^2)^{2}} {{(\lambda_1 - \lambda_2 )^2}} \beta \frac{{g'}}{g}u_2 u_{22}]  \notag \\
=& \frac{- 2x_2(r^2  - x_1^2)}{\eta} [1+ O (\frac{1} {{\eta \lambda_1 }})],
\end{align}
which implies
\begin{align}\label{3.19}
\frac{{\eta _2 }}{{\eta }}= \frac{- 2x_2(r^2  - x_1^2)}{\eta} [1+ O (\frac{1} {{\eta \lambda_1 }})].
\end{align}
That is $ \frac{{\eta _2 }}{{\eta  }} \approx \frac{- 2x_2(r^2  - x_1^2)}{\eta}$ if $\eta \lambda_1$ is big enough.
Hence
\begin{align}\label{3.20}
\beta f (1+ |\nabla u|^2)[(1+ |\nabla u|^2)\frac{g'}{g} -4 ]\frac{{\eta_{2} }}{{\eta }} u_2  = O (\frac{1} {{ \eta }}).
\end{align}

Taking second derivatives of $\eta$, we can get
\begin{align*}
\eta _{ii}  =& [ - 2 + 2\left\langle {x,\tau } \right\rangle \left\langle {x,\tau } \right\rangle _{ii}  + 2\left\langle {x,\tau } \right\rangle _i \left\langle {x,\tau } \right\rangle _i ][r^2  - \left\langle {x,\tau } \right\rangle ^2 ]  \\
&+ 2[ - 2x_i  + 2\left\langle {x,\tau } \right\rangle \left\langle {x,\tau } \right\rangle _i ][ - 2\left\langle {x,\tau } \right\rangle \left\langle {x,\tau } \right\rangle _i ]  \\
&+ [r^2  - |x|^2  + \left\langle {x,\tau } \right\rangle ^2 ][ - 2\left\langle {x,\tau } \right\rangle \left\langle {x,\tau } \right\rangle _{ii}  - 2\left\langle {x,\tau } \right\rangle _i \left\langle {x,\tau } \right\rangle _i ]  \\
=& [ - 2 + 2x_1 \left\langle {x,\tau } \right\rangle _{ii}  + 2(\delta_{i1} +  \left\langle {x, \partial_i \tau } \right\rangle)^2 ][r^2  - x_1 ^2 ]  \\
&+ 2[ - 2x_i  + 2x_1 (\delta_{i1} +  \left\langle {x, \partial_i \tau } \right\rangle) ][ - 2x_1 (\delta_{i1} +  \left\langle {x, \partial_i \tau } \right\rangle) ] \\
&+ (r^2 -x_2^2)[ - 2x_1 \left\langle {x,\tau } \right\rangle _{ii}  - 2(\delta_{i1} +  \left\langle {x, \partial_i \tau } \right\rangle)^2 ],
\end{align*}
so
\begin{align}
\label{3.21}\eta _{11}  =& - 2(r^2-x_2^2) - 2x_1 (x_1^2 -x_2^2 )\left\langle {x,\tau } \right\rangle _{11} \notag \\
&+ ( 4x_2^2 -12 x_1^2 )\left\langle {x, \partial_1 \tau } \right\rangle + ( 2x_2^2 -10 x_1^2 )  \left\langle {x, \partial_1 \tau } \right\rangle^2,  \\
\label{3.22}\eta _{22}  =& - 2(r^2  - x_1 ^2 ) - 2x_1 (x_1^2 -x_2^2 )\left\langle {x,\tau } \right\rangle _{22} \notag \\
&+ 8 x_1 x_2 \left\langle {x, \partial_2 \tau } \right\rangle + ( 2x_2^2 -10 x_1^2 )  \left\langle {x, \partial_2 \tau } \right\rangle^2.
\end{align}
Hence
\begin{align}\label{3.23}
\beta [\lambda_2 \frac{{\eta _{11}}}{\eta } + \lambda_1 \frac{{\eta _{22}
}}{\eta } ] =& -2\beta [\lambda_2 \frac{{r^2-x_2^2}}{\eta } + \lambda_1 \frac{{r^2-x_1^2}}{\eta } ]  \notag\\
&-2\beta \frac{{x_1 (x_1^2 -x_2^2 )}}{\eta } [\lambda_2 \left\langle {x,\tau } \right\rangle _{11} + \lambda_1 \left\langle {x,\tau } \right\rangle _{22}]  \notag \\
&+\beta \lambda_2  [ \frac{{x_2 ( 4x_2^2 -12 x_1^2 ) }}{\eta }\frac{u_{112}} {{\lambda_1 - \lambda_2 }}+\frac{{x_2^2 ( 2x_2^2 -10 x_1^2 )}}{\eta }(\frac{u_{112}} {{\lambda_1 - \lambda_2 }})^2 ] \\
& +\beta \lambda_1 [ \frac{{8x_1 x_2^2 }}{\eta } \frac{u_{221}} {{\lambda_1 - \lambda_2 }} + \frac{{ x_2^2( 2x_2^2 -10 x_1^2 ) }}{\eta }(\frac{u_{221}} {{\lambda_1 - \lambda_2 }})^2  ]. \notag
\end{align}
Direct calculations yield
\begin{align*}
 \left\langle {x,\tau } \right\rangle _{11}  =& \frac{{\partial ^2 }}
{{\partial x_1 ^2 }}[\sum\limits_{m = 1}^2 {x_m \tau _m } ] = 2\frac{{\partial \tau _1 }}
{{\partial x_1 }} + \sum\limits_{m = 1}^2 {x_m \frac{{\partial ^2 \tau _m }}
{{\partial x_1 ^2 }}}  \\
=& 2\frac{{\partial \tau _1 }}
{{\partial u_{pq} }}u_{pq1}  + \sum\limits_{m = 1}^2 {x_m [\frac{{\partial \tau _m }}
{{\partial u_{pq} }}u_{pq11}  + \frac{{\partial ^2 \tau _m }}
{{\partial u_{pq} \partial u_{rs} }}u_{pq1} u_{rs1} ]}  \\
=& 0+x_1 \frac{{\partial ^2 \tau _1 }}{{\partial u_{pq} \partial u_{rs} }}u_{pq1} u_{rs1} +x_2 [\frac{{\partial \tau _2 }}
{{\partial u_{pq} }}u_{pq11}  + \frac{{\partial ^2 \tau _2 }}
{{\partial u_{pq} \partial u_{rs} }}u_{pq1} u_{rs1} ] \\
=&  - x_1\left[ {\frac{{u_{112} }} {{\lambda _1  - \lambda _2 }}} \right] ^2 + x_2 \left[\frac{{ 1}} {{\lambda _1  - \lambda _2 }} \right] u_{1211}+ 2x_2 \left[-\frac{{ u_{112} u_{111} }} {{(\lambda _1  - \lambda _2)^2 }} +\frac{{ u_{112} u_{221} }} {{(\lambda _1  - \lambda _2)^2 }} \right],
\end{align*}
Similarly, we have
\begin{align*}
\left\langle {x,\tau } \right\rangle _{22}  =& \frac{{\partial ^2 }}
{{\partial x_2 ^2 }}[\sum\limits_{m = 1}^2 {x_m \tau _m } ] = 2\frac{{\partial \tau _2 }}
{{\partial x_2 }} + \sum\limits_{m = 1}^2 {x_m \frac{{\partial ^2 \tau _m }}
{{\partial x_2 ^2 }}}  \\
=& 2\frac{{\partial \tau _2 }}
{{\partial u_{pq} }}u_{pq2}  + \sum\limits_{m = 1}^2 {x_m [\frac{{\partial \tau _m }}
{{\partial u_{pq} }}u_{pq22}  + \frac{{\partial ^2 \tau _m }}
{{\partial u_{pq} \partial u_{rs} }}u_{pq2} u_{rs2} ]}  \\
=& 2\left[  \frac{{1 }}
{{\lambda _1  - \lambda _2}} \right]u_{221}  - x_1\left[ {\frac{{u_{221} }} {{\lambda _1  - \lambda _2 }}} \right] ^2 \\
 &+ x_2 \left[\frac{{ 1}} {{\lambda _1  - \lambda _2 }} \right] u_{1222}+ 2x_2 \left[-\frac{{ u_{112} u_{221} }} {{(\lambda _1  - \lambda _2)^2 }} +\frac{{ u_{222} u_{221} }} {{(\lambda _1  - \lambda _2)^2 }} \right],
\end{align*}
then
\begin{align}\label{3.24}
&\lambda_2 \left\langle {x,\tau } \right\rangle _{11} + \lambda_1 \left\langle {x,\tau } \right\rangle _{22}\notag \\
=&  - x_1 \lambda_2 \left[ {\frac{{u_{112} }} {{\lambda _1  - \lambda _2 }}} \right] ^2 + 2\lambda_1\left[  \frac{{ u_{221}}}
{{\lambda _1  - \lambda _2}} \right] - x_1 \lambda_1\left[ {\frac{{u_{221} }} {{\lambda _1  - \lambda _2 }}} \right] ^2 \notag \\
&+ x_2 \left[\frac{{ 1}} {{\lambda _1  - \lambda _2 }} \right] \left[ - \frac{u_{111}}{u_{11}} [f_2 (1+ |\nabla u|^2)^{2} + f \cdot 8(1+ |Du|^2) u_2 u_{22} ] \right. \notag \\
&\qquad \qquad \qquad \qquad \left.+  \frac{u_{112}}{u_{11}}[f_1 (1+ |\nabla u|^2)^{2} + f \cdot 8(1+ |\nabla u|^2) u_1 u_{11}] + O(\lambda_{1})\right] \notag \\
&- 2x_2\frac{{ u_{112} }} {{(\lambda _1  - \lambda _2)^2 }}\left[f_1 (1+ |\nabla u|^2)^{2} + f \cdot 2(1+ |\nabla u|^2) \cdot 2 u_1 u_{11} \right] \notag \\
&+2x_2\frac{{ u_{221} }} {{(\lambda _1  - \lambda _2)^2 }}\left[f_2 (1+ |\nabla u|^2)^{2} + f \cdot 2(1+ |\nabla u|^2) \cdot 2 u_2 u_{22} \right] \notag \\
=&  \left[ {\frac{{u_{112} }} {{\lambda _1  - \lambda _2 }}} \right] ^2 \cdot O(\frac{1}{\lambda_1}) + \left[  \frac{{ u_{112}}}
{{\lambda _1  - \lambda _2}} \right]\cdot O(\frac{1}{\lambda_1}) + O(1) \notag \\
&+\left[ {\frac{{u_{221} }} {{\lambda _1  - \lambda _2 }}} \right] ^2 \cdot O(\lambda_1)+\left[ {\frac{{u_{221} }} {{\lambda _1  - \lambda _2 }}} \right]\cdot O(\lambda_1) + \frac{u_{111}}{u_{11}} \cdot O(\frac{1}{\lambda_1}).
\end{align}

From \eqref{3.23} and \eqref{3.24}, we can get
\begin{align}\label{3.25}
&\beta [\lambda _2\frac{{\eta _{11}}}{\eta } + \lambda_1 \frac{{\eta _{22} }}{\eta } ] \notag\\
=&O(\frac{{1}}{\eta\lambda _1 } ) -2\beta \lambda_1 \frac{{r^2-x_1^2}}{\eta } + O(\frac{{1}}{\eta} )\notag \\
&+(\frac{u_{112}} {{\lambda_1 - \lambda_2 }})^2  \cdot O(\frac{{1}}{\eta\lambda _1 } ) +(\frac{u_{112}} {{\lambda_1 - \lambda_2 }})  \cdot O(\frac{{1}}{\eta\lambda _1 } )  + \frac{u_{111}}{u_{11}} \cdot O(\frac{1}{\eta \lambda_1})\notag\\
& +(\frac{u_{221}} {{\lambda_1 - \lambda_2 }})^2 \cdot O(\frac{{\lambda_1}}{\eta })+\frac{u_{221}} {{\lambda_1 - \lambda_2 }}  \cdot O(\frac{{\lambda_1}}{\eta }) \notag\\
=& -2\beta \lambda_1 \frac{{r^2-x_1^2}}{\eta } + O(\frac{{1}}{\eta} )+(\frac{u_{112}} {{\lambda_1 - \lambda_2 }})^2  \cdot O(\frac{{1}}{\eta\lambda _1 } ) +(\frac{u_{112}} {{\lambda_1 - \lambda_2 }})  \cdot O(\frac{{1}}{\eta\lambda _1 } )\notag \\
&  + \frac{u_{111}}{u_{11}} \cdot O(\frac{1}{\eta \lambda_1})+[ -\frac{f (1+ |\nabla u|^2)^{2}}{\lambda_1} \frac{u_{111}}{u_{11}} + O(1)]^2 \cdot O(\frac{{1}}{\eta \lambda_1})\notag\\
&+[ -\frac{f (1+ |\nabla u|^2)^{2}}{\lambda_1} \frac{u_{111}}{u_{11}} + O(1)]  \cdot O(\frac{{1}}{\eta }) \notag\\
\geq& -2\beta \lambda_1 \frac{{r^2-x_1^2}}{\eta } + O(\frac{{1}}{\eta} )- \frac{\lambda_2}{2}(\frac{u_{111}}{u_{11}})^2 -\frac{\lambda_1}{4}(\frac{u_{112}} {{u_{11} }})^2.
\end{align}

Now we just need to estimate $-2\beta \lambda_1 \frac{{r^2-x_1^2}}{\eta }$. If $x_2^2 \leq \frac{r^2}{2}$, we can get
\begin{align*}
-2\beta \lambda_1 \frac{{r^2-x_1^2}}{\eta }= - \frac{{8}}{r^2-x_2^2 } \lambda_1 \geq - \frac{{16}}{r^2} \lambda_1 \geq - \frac{1}{2}\frac{{c_0}}{r^2} f\lambda_1 = -  \frac{1}{2}\frac{{g'}}{g} f\lambda_1.
\end{align*}
If $x_2^2 > \frac{r^2}{2}$, we can get
\begin{align*}
-2\beta \lambda_1 \frac{{r^2-x_1^2}}{\eta }=& - \frac{{8}}{r^2-x_2^2 } \lambda_1 \geq -\frac{{x_2^2}}{r^2-x_2^2 } \frac{{8}}{r^2-x_2^2 } \lambda_1 = -\beta \lambda_1 \frac{1}{2}[\frac{{2x_2}}{r^2-x_2^2 }]^2  \\
\geq& -\beta \lambda_1 [\frac{\eta_2}{\eta}]^2,
\end{align*}
if $\eta \lambda_1$ is big enough. Hence
\begin{align}\label{3.26}
-2\beta \lambda_1 \frac{{r^2-x_1^2}}{\eta }\geq  -  \frac{1}{2}\frac{{g'}}{g} f\lambda_1-\beta \lambda_1 [\frac{\eta_2}{\eta}]^2,
\end{align}
and
\begin{align}\label{3.27}
\beta [\lambda_2\frac{{\eta _{11}}}{\eta } + \lambda_1 \frac{{\eta _{22}
}}{\eta } ] \geq -  \frac{1}{2}\frac{{g'}}{g} f\lambda_1-\beta \lambda_1 [\frac{\eta_2}{\eta}]^2- \frac{\lambda_2}{2}(\frac{u_{111}}{u_{11}})^2 -\frac{\lambda_1}{4}(\frac{u_{112}} {{u_{11} }})^2+ O(\frac{{1}}{\eta} ).
\end{align}

\end{proof}

Now we continue to prove Theorem \ref{th1.1}.  From \eqref{3.11} and Lemma \ref{lem3.1}, we can get
\begin{align}
0 \ge \sum\limits_{i = 1}^2 {F^{ii} \varphi _{ii} }  \geq&  \frac{1}{2}\frac{{g'}}{g} f \lambda_1  +O (\frac{{1}}{\eta}) +O(1) \notag \\
=& \frac{1}{2}\frac{{g'}}{g} f \lambda_1  - \frac{{C_0}}{\eta} -C_0,
\end{align}
So we can get
\begin{align} \label{3.29}
\eta \lambda_1 \leq&   C_1.
\end{align}
where $C_0$ and $C_1$ are positive constants depending only on $m$, $M$, $R$, $\sup |\nabla f|$, $\sup |\nabla^2 f|$, and $\sup |\nabla u|$. So we can easily get
\begin{align*}
u_{\tau(0) \tau(0)}(0) \leq& \frac{1}{r^{4 \beta}} \phi (0)\leq \frac{1}{r^{4\beta}} \phi (x_0)  \leq  C,
\end{align*}
and
\begin{align} \label{3.30}
|u_{\xi \xi} (0) |\leq u_{\tau(0) \tau(0)}(0) \leq  C, \quad \forall ~\xi \in \mathbb{S}^{1}.
\end{align}
Then we have proved \eqref{1.3}, and Theorem \ref{th1.1} holds.

\textbf{Acknowledgement}.
The authors would like to express sincere gratitude to Prof. Xi-Nan Ma for the constant encouragement in this subject.


\begin{thebibliography}{99}

\bibitem{A07}
B. Andrews, Pinching estimates and motion of hypersurfaces by curvature functions, J. Reine Angew. Math., 608(2007), 17-33.

\bibitem{CHO15}
C.Q. Chen, F. Han, Q.Z. Ou, The interior $C^2$ estimate for Monge-Amp\`{e}re equation in dimension $n=2$, arXiv:1512.01146, 2015.

\bibitem{CW01}
K.S. Chou, X.J. Wang, A variation theory of the Hessian equation, Comm. Pure Appl. Math., 54(9)(2001), 1029-1064.

\bibitem{GRW14}
P. Guan, C. Ren, Z. Wang,  Global $C^2$ estimates for convex solutions of curvature equations, Comm. Pure Appl. Math, 68(2015), 1287-1325.

\bibitem{H59}
E. Heinz, On elliptic Monge-Amp\`{e}re equations and Weyl¡¯s embedding problem, J. Analyse Math., 7(1959), 1-52.

\bibitem{LRW15}
M. Li, C. Ren, Z. Wang, An interior estimate for convex solutions and a rigidity theorem, J. Funct. Anal., 270(2016), 2691-2714.

\bibitem{P78}
A.V. Pogorelov, The multidimensional Minkowski problem, New York, Wiley 1978.

\bibitem{T97}
N.S. Trudinger, Weak solutions of Hessian equations, Comm. Partial Differential
Equations, 22(7-8)(1997), 1251-1261.

\bibitem{U90}
J. Urbas, On the existence of nonclassical solutions for two class of fully nonlinear
elliptic equations, Indiana Univ. Math. J., 39(1990), 355-382.

\bibitem{WY09}
 M. Warren, Y. Yuan, Hessian estimates for the $\sigma_2$ equation in dimension 3, Comm. Pure Appl. Math., 62(2009), 305-321.

\end{thebibliography}
\end{document}